\documentclass[10pt]{article}

\usepackage{xcolor}

\usepackage[
naturalnames,
bookmarks=false
]{hyperref}

\usepackage{amsfonts,amssymb}
\usepackage{amsmath}
\usepackage{mathrsfs}
\usepackage{latexsym}
\usepackage{amscd}

\input{xy} \usepackage[all]{xy}
\xyoption{all} \xyoption{poly}

\DeclareSymbolFont{AMSb}{U}{msb}{m}{n}
\DeclareMathSymbol{\N}{\mathbin}{AMSb}{"4E}
\DeclareMathSymbol{\Z}{\mathbin}{AMSb}{"5A}
\DeclareMathSymbol{\R}{\mathbin}{AMSb}{"52}
\DeclareMathSymbol{\Q}{\mathbin}{AMSb}{"51}
\DeclareMathSymbol{\I}{\mathbin}{AMSb}{"49}
\DeclareMathSymbol{\C}{\mathbin}{AMSb}{"43}

\newcommand{\iso}{\cong}

\newcommand{\id}{\textnormal{id}}

\newcommand{\Hom}{\textnormal{Hom}}

\newcommand{\End}{\textnormal{End}}

\newcommand{\tn}[1]{\textnormal{#1}}
\newcommand{\cat}[1]{\tn{\textbf{#1}}}

\begin{document}


\newtheorem{thm}{Theorem}[section]
\newtheorem{cor}[thm]{Corollary}
\newtheorem{lem}[thm]{Lemma}
\newtheorem{pro}[thm]{Proposition}
\newtheorem{defi}[thm]{Definition}
\newtheorem{rem}[thm]{Remark}
\newtheorem{exa}[thm]{Example}
\newenvironment{proof}{\noindent \textbf{{Proof.}} \sf}

\def\contentsname{Index} \def\refname{References}

\def\qed{\hfill $\diamond$ \bigskip}

\newcommand\CC{{\mathbb{C}}} \newcommand\RR{{\mathbb{R}}} \newcommand\QQ{{\mathbb{Q}}} \newcommand\ZZ{{\mathbb{Z}}} \newcommand\NN{{\mathbb{N}}}

\def\A{{\mathcal A}} \def\B{{\mathcal B}} \def\C{{\mathcal C}} \def\D{{\mathcal D}} \def\E{{\mathcal E}} \def\F{{\mathcal F}} \def\I{{\mathcal I}} \def\J{{\mathcal J}} \def\K{{\mathcal K}} \def\L{{\mathcal L}} \def\U{{\mathcal U}} \def\Z{{\mathcal Z}} \def\M{{\mathcal M}} \def\N{{\mathcal N}}

\def\G{{\mathcal G}}
\def\H{{\mathcal H}}

\def\place{{-}}

\def\lim{\mathop{\rm lim}\nolimits} \def\min{\mathop{\rm min} \nolimits} \def\max{\mathop{\rm max} \nolimits} \def\Hom{\mathop{\sf Hom}\nolimits} \def\rad{\mathop{\rm rad}\nolimits} \def\Rad{\mathop{\rm rad}\nolimits} \def\End{\mathop{\rm End}\nolimits} \def\ann{\mathop{\rm ann}\nolimits} \def\Ext{\mathop{\sf Ext}\nolimits} \def\Tor{\mathop{\rm Tor}\nolimits} \def\soc{\mathop{\rm soc}\nolimits} \def\op{\mathop{\rm op}\nolimits} \def\Ker{\mathop{\rm Ker}\nolimits}  \def\Coker{\mathop{\rm Coker}\nolimits} \def\coker{\mathop{\rm Coker}\nolimits} \def\Im{\mathop{\rm Im}\nolimits} \def\dimQ{\mathop{\rm dim_\Q}\nolimits} \def\rank{\mathop{\rm rank}\nolimits} \def\Out{\mathop{\rm Out}\nolimits} \def\Int{\mathop{\rm Int}\nolimits} \def\mod{\mathop{\sf mod}\nolimits} \def\top{\mathop{\rm top}\nolimits} \def\Mod{\mathop{\rm Mod}\nolimits} \def\ind{\mathop{\rm ind}\nolimits} \def\add{\mathop{\rm add}\nolimits} \def\Add{\mathop{\rm Add}\nolimits} \def\tr{\mathop{\rm tr}\nolimits} \def\lamod{\mbox{{\rm mod}$\Lambda$}} \def\modla{\mbox{{\rm mod}$\Lambda$}} \def\laMod{\mbox{{\rm Mod}$\Lambda$}} \def\Fix{\mbox{{\rm Fix}}} \def\HH{\mbox{{\rm HH}}} \def\H{\mbox{{\rm H}}}

\def\ootimes{\mbox{\footnotesize$\otimes$}} \def\aut{\cat{Aut}} \def\Aut{\mathop{\bf Aut}\nolimits} \def\st{\mathsf{St}} \def\mor{\mathsf{Mor}}
\def\gal{\mathsf{Gal}} \def\cov{\mathsf{Cov}} \def\pr{\mathsf{pr}}
\def\galcov{\cat{GalCov}}
\def\ker{\mathop{\rm Ker}\nolimits}
\def\im{\mathop{\rm Im}\nolimits}


\sf

\title{Gradings, smash products and Galois coverings of a small category} \author{Claude Cibils and John MacQuarrie}

\date{}

\maketitle

\begin{abstract}
In this paper we develop the theory of coverings of a small connected category $\B$.  We show that the category of Galois coverings of $\B$ is equivalent to the category of Galois coverings of its fundamental groupoid.  Making use of effective gradings of $\B$ we explicitly construct Galois coverings through a smash product analogous to the one considered in the linear case, see \cite{CM,ciso}.  In particular, the universal  cover of $\B$ can be obtained from its fundamental groupoid.
\end{abstract}


\section{\sf Introduction}

In this paper we consider a small connected category $\B$ and its category of Galois coverings.  A main purpose is to relate this category with other categories, in particular with ``effective gradings'' of $\B$.

The results we obtain are inspired by but are different from those in \cite{CRS,CRS2} for linear categories over a field $k$, namely  enriched categories over $k$-vector spaces. An important difference relies on the existence of groupoids in the context of small categories, that is categories where all the morphisms are invertible. There are no analogous structures available when considering $k$-categories.

We first recall in Section \ref{groupoid} the construction by P. Gabriel and M. Zisman \cite{GAZI} of a canonical groupoid associated to $\B$ - its ``category of fractions'' $Q\B$, whose objects are the same as the objects of $\B$.  In doing so, we make use of $\B$-free categories over a graph.  The category $Q\B$ is the quotient of a $\B$-free category by certain equivalence relations on the morphisms.  The category of fractions provides a functor from small categories to groupoids.  We make use of the canonical functor $Q_{\B}:\B\to Q\B$, which is the identity on objects but which in general is neither full nor faithful.

In \cite{QU}, D. Quillen defines the fundamental group of a small connected category $\B$ to be the fundamental group of a topological space closely related to $\B$, namely the geometric realization of its nerve.  P. Gabriel - M. Zisman and D. Quillen (\cite{GAZI,QU}) proved that the automorphism group of any chosen object of $Q\B$ is isomorphic to the fundamental group of $\B$ as defined.  In this paper we consider $Q\B$ to be the fundamental groupoid of $\B$, avoiding the choice of a base-point.

In Section \ref{coverings} we first recall the definition of a covering of a category and of a Galois covering, as well as some important properties.  Our first result is that $Q$ provides a functor from Galois coverings of $\B$ to Galois coverings of $Q\B$ which is an equivalence of categories.  For this purpose we give a canonical form for a Galois covering, namely any Galois covering is isomorphic to an orbit category by the free action of a group.

Coverings and universal covers of groupoids are considered by J.P. May (see \cite{May}), hence  the existence of a universal Galois covering of $\B$ is inferred from the above equivalence.  Note that universal coverings of $\B$ are known to exist by work of N. Ojeda Bar \cite{OJ} and K. Tanaka \cite{ta}, whose methods are closely related to one another but differ from ours.

For a $k$-category over a field $k$ presented by generators and relations, a theory of coverings has been developed (see \cite{bo,boga,ga,MP,le,le1,le2}).  By making use of linear group gradings of a $k$-category, all Galois coverings of the category can be described in an intrinsic way, see \cite{CRS,CRS2,CM}.  A fundamental group ``\`{a} la Grothendieck'' is associated to each $k$-category in much the same way as the fundamental group is used in algebraic geometry by A. Grothendieck and C. Chevalley (see for instance \cite{dodo}).  Note however that for a $k$-category there need not be a universal cover, nor a $k$-category of fractions.

For an abstract small category $\B$, a grading by a group $G$ is a disjoint union decomposition of each morphism set of $\B$ indexed by elements of $G$ such that composition in $\B$ agrees with group multiplication in $G$.  This is clearly equivalent to a functor from $\B$ to a category with one object and morphism set $G$.  Note that in case a $k$-category is graded by a group the situation is different. Indeed in this context associating to a morphism the formal sum of the degrees of its homogeneous components does not provide a linear functor from $\B$ to the single object category which has the group algebra as endomorphisms.

More generally, we consider in Section \ref{gradings} groupoid gradings of $\B$.  That is, functors from $\B$ to groupoids.  We introduce a smash product construction in this context which is an instance of a comma category, see \cite[II.6]{MacLane}.  We use it to obtain a covering of $\B$ from a groupoid grading of $\B$.  In the case where the grading is bijective on objects, we demonstrate that the covering associated to a grading is connected if and only if the factorization of the grading through the fundamental groupoid $Q\B$ of $\B$ is full (such gradings will be called \emph{effective}).  We prove that every Galois covering is isomorphic to a smash product by an effective grading.

Finally in Section \ref{effective} we consider the category of \emph{pointed Galois coverings of $\B$} as follows.  First a fixed base object $b_0$ of $\B$ is chosen. Then the objects of this category are Galois coverings of $\B$ provided with an object in the fibre of $b_0$, and the morphisms are those maps of coverings respecting the fibre objects.  We obtain another main result, namely the category of effective gradings of $\B$ is equivalent to the category of pointed Galois coverings of $\B$.

It follows that the smash product allows one to construct Galois coverings of a small category.  Hence we obtain an explicit description of the universal cover of $\B$ via the functor $Q_{\B}:\B\to Q\B$ -- as expected, the fundamental groupoid is an effective grading which is universal.  We end Section \ref{effective} with an example, using Cayley graphs to give an explicit description of the universal cover of the category consisting of two objects $x_0$ and $x$ and an arbitrary set of morphisms $E$ from $x$ to $x_0$.

We thank Andrea Solotar for an attentive reading of a preliminary version of this paper and for several improving suggestions.  We thank Clemens Berger for interesting discussions.  Thanks also to Jeremy Rickard, whose topological insights were of great help to the second author.

\section{The fundamental groupoid}\label{groupoid}

All categories considered in this paper are small. If $\B$ is a category the set of objects is denoted $\B_0$ and the set of morphisms is denoted $\B_1$.  The set of morphisms from  $x$ to  $y$ is denoted ${}_y\B_x$. The \textbf{source} object $s(\beta)$ of a morphism $\beta$ in ${}_y\B_x$ is $x$ while its \textbf{target} object $t(\beta)$ is $y$.

In order to consider the fundamental groupoid associated to $\B$ we recall the \textbf{category of fractions }considered by P. Gabriel and M. Zisman \cite{GAZI}, see also \cite{GEMA,KASC}. We provide a construction for the convenience of the reader and for further explicit use.

A directed graph $\Z$ is given by two sets $\Z_0$ (objects) and $\Z_1$ (arrows) as with a category.  In this context a morphism between directed graphs is like a functor between categories but without any requirement concerning composition, see \cite[II.7]{MacLane}.

Given a category $\B$ we denote by $\mathcal{B}_1^\bullet$ the set of morphisms of $\mathcal{B}$ that are not identities.

Let $\mathcal{B}$ be a category and let $\Z$ be a directed graph having the same objects as $\B$.

\begin{defi}\label{free}
The $\mathcal{B}$-\textbf{free category} over $\mathcal{Z}$, denoted  $\mathcal{F}_{\mathcal{B}}\Z$, is defined as follows.  As objects {$\left(\mathcal{F}_{\mathcal{B}}\Z\right)_0=\mathcal{B}_0$}.  A morphism $\delta$ in { ${}_y\left(\mathcal{F}_{\mathcal{B}}\Z\right)_x$ }is a sequence $(\delta_n,\hdots,\delta_1)$, where
\begin{itemize}
\item Each $\delta_i$ is an element of $\mathcal{B}_1^\bullet$ or $\Z$.
\item The source of $\delta_1$ is $x$ and the target of $\delta_n$ is $y$.  The empty sequence is only allowed in case $x=y$.
\item For each $i\in \{1,\hdots,n-1\}$ the target of $\delta_i$ is the source of $\delta_{i+1}$.
\item For each $i$, if $\delta_i$ is in $\mathcal{B}_1^\bullet$ then $\delta_{i+1}$ is in $\Z$. 
\end{itemize}
The composition $(\delta_n,\hdots,\delta_1)(\gamma_m,\hdots,\gamma_1)$ is given by concatenation unless $\gamma_m$ and $\delta_1$ are both in $\mathcal{B}_1^\bullet$. In this case composition is as follows: if $\delta_1\gamma_m\neq \id$  then
the result is $(\delta_n,\hdots,\delta_1\gamma_m,\hdots,\gamma_1)$.

If $\delta_1\gamma_m = \id$ then the composition is $(\delta_n,\hdots,\delta_2,\gamma_{m-1},\hdots,\gamma_1)$.
In this last case, if $m=n=1$ we obtain the empty sequence.
\end{defi}

The identity morphism at an object $x$ is given by the empty sequence with source and target $x$.  The reader might wish to compare this definition with that of the free category on a graph (see \cite[II.7]{MacLane}).  There is an obvious functor $L:\mathcal{B}\to \mathcal{F}_{\mathcal{B}}\Z$ that is the identity on objects, sends a non identity morphism $\beta$ to the sequence $(\beta)$ and an identity to the corresponding empty sequence.

\begin{pro}
The $\mathcal{B}$-free category $\mathcal{F}_{\mathcal{B}}\Z$ together with the functor $L$ has the following universal property.  Given $F:\mathcal{B}\to \mathcal{C}$ a functor and a morphism $\theta$ of directed graphs from $\Z$ to the underlying directed graph of $\C$ that coincides with $F$ on objects, there is a unique functor $F_\theta:\mathcal{F}_{\mathcal{B}}\Z\to \mathcal{C}$ such that $F_\theta L=F$ and such that $F_\theta(\gamma)=\theta(\gamma)$ for all $\gamma\in \Z$.
\end{pro}

\begin{proof}
On objects $F_\theta$ coincides with $F$ while the value of $F_\theta$ on a sequence $(\delta_n,\hdots,\delta_1)$ is the composition of the values of the $\delta_i$ under the maps $F$ or $\theta$ as appropriate.  Using the definition of composition given above, this defines a functor. \qed
\end{proof}

An \textbf{ideal relation} (also called a congruence relation) on $\mathcal{B}$ is an equivalence relation $\sim$ on each morphism set having the additional property that $\alpha\sim\beta$ implies $\alpha\gamma\sim\beta\gamma$ and $\gamma\alpha\sim\gamma\beta$ whenever these compositions make sense.  Given an ideal relation $\sim$ the \textbf{quotient category} $\mathcal{B}/\sim$ is the category with objects $\mathcal{B}_0$ and morphisms equivalence classes under $\sim$.  Note that the inherited composition is well-defined.

Let $R$ be a subset of $\B_1\times\B_1$ with the property that whenever $(\alpha,\beta)\in R$ then $\alpha$ and $\beta$ share source and target objects.  The \textbf{ideal relation generated by $R$}, denoted $\sim_R$, is the smallest ideal relation containing $R$.  Thus $\sim_R$ is the smallest ideal relation of $\B$ having the property that $\alpha\sim_R\beta$ whenever $(\alpha,\beta)\in R$.  The ideal relation generated by $R$ exists, being the intersection of the ideal relations containing $R$.  We retain this notation for the following:

\begin{pro}\label{respecting pairs means respecting relations}
Let $F:\mathcal{B}\to \mathcal{C}$ be a functor respecting $R$ in the sense that $(\alpha,\beta)\in R$ implies $F(\alpha)=F(\beta)$.  Then $F$ respects $\sim_R$, meaning $\alpha\sim_R\beta$ implies $F(\alpha)=F(\beta)$.
\end{pro}

\begin{proof}
The relation $\sim_F$ given by $\alpha\sim_F\beta$ if and only if $F(\alpha)=F(\beta)$ is an ideal relation containing $R$.  Since $\sim_R$ is the smallest such ideal relation it is contained in $\sim_F$.  It follows that $F$ respects $\sim_R$. \qed
\end{proof}

\begin{defi}
Let $\mathcal{B}$ be a small category and $S$ a subset of $\mathcal{B}_1$.  Let $S'$ be a formal copy of $S$.  For each $\beta'\in S'$, define the source of $\beta'$ to be the target of $\beta$ and the target of $\beta'$ to be the source of $\beta$. This way we consider $S'$ as a directed graph with objects $\B_0$. The \textbf{$S$-category of fractions} $\mathcal{B}_S$ (also denoted $\B[S^{-1}])$) is the category $\mathcal{F}_{\mathcal{B}}S'/\sim_S$, where $\sim_S$ is the ideal relation generated by the pairs $(\beta\beta',\id_{t(\beta)})$ and $(\beta'\beta,\id_{s(\beta)})$ for each $\beta\in S$.
\end{defi}

Note that there is a canonical functor $Q_S:\mathcal{B}\to \mathcal{B}_S$ given by the composition $\mathcal{B}\to \mathcal{F}_{\mathcal{B}}S'\to \mathcal{B}_S$.  Note also that $Q_S(\beta)$ is invertible for each $\beta\in S$.

\begin{thm}\label{UP of Q}
The category $\mathcal{B}_S$ together with the functor $Q_S:\mathcal{B}\to \mathcal{B}_S$ satisfies the following universal property: whenever $F:\mathcal{B}\to \mathcal{C}$ is a functor such that $F(\beta)$ is invertible for each $\beta\in S$, there is a unique functor $F':\mathcal{B}_S\to \mathcal{C}$ such that $F=F' Q_S$.
\end{thm}

\begin{proof}
Define the map $\eta:S'\to \mathcal{C}_1$ by $\beta'\mapsto F(\beta)^{-1}$.  By the universal property of the $\mathcal{B}$-free category, we have a unique functor $F_\eta:\mathcal{F}_{\mathcal{B}}S'\to \mathcal{C}$ that agrees with $F$ on $\mathcal{B}$ and with $\eta$ on $S'$.  Note that ${F}_\eta$ respects the generating pairs $(\beta\beta',\textnormal{id})$ and $(\beta'\beta,\textnormal{id})$, so by \ref{respecting pairs means respecting relations}, ${F}_\eta(\alpha)={F}_\eta(\beta)$ whenever $\alpha\sim_S\beta$.  Thus  ${F}_\eta$ has the required property and is clearly unique. \qed
\end{proof}

When $S=\B_1$, the category $\mathcal{B}_{\B_1}$ will be denoted  $Q\B$ and the canonical functor $\B\to Q\B$ will be denoted $Q_{\B}$.  Note that every morphism of $Q\B$ is invertible.  That is, $Q\B$ is a {groupoid}.  In this case a sequence as considered in Definition \ref{free} is called a \textbf{walk}.

A category is said to be $\textbf{connected}$ if its underlying undirected graph has one connected component.  Note that $Q\mathcal{B}$ is connected if and only if $\mathcal{B}$ is connected.  In \cite{QU} D.Quillen defines the fundamental group $\pi_1(\mathcal{B})$ of a small connected category $\mathcal{B}$ to be the fundamental group of a related topological space, namely the geometric realization of a simplicial set known as the \emph{nerve} of $\B$.  See also the work by G. Minian \cite{MIN} where the fundamental group is obtained through functors from interval categories and the paper by K. Tanaka \cite{ta} where a groupoidification of a category is considered.

In the quoted paper D. Quillen proves that the fundamental group of $\mathcal{B}$ he defines is isomorphic to the automorphism group of any object of $Q\mathcal{B}$.  Since $Q\B$ is a connected groupoid, it is equivalent to the full subcategory of $Q\B$ generated by any single object.  The endomorphisms of an object form a group that is isomorphic to $\pi_1(\B)$, and one thus obtains non-canonical functors $\B\to \pi_1(\B)$ by composing $Q_{\B}$ with an appropriate equivalence.  In order to avoid the difficulties arising from this choice of object it is natural to consider the entire $Q\B$ as the fundamental groupoid of $\B$. This point of view agrees with the analysis in the Introduction of \cite{ta}.

\begin{defi}
Let $\mathcal{B}$ be a small category.  The \textbf{fundamental groupoid} of $\mathcal{B}$ is the category of fractions $Q\mathcal{B}$.
\end{defi}

\begin{rem}
In case $\mathcal{M}$ is a monoid thought of as a category with one object then $Q\mathcal{M}$ is its group of fractions (also called the \emph{universal group} of $\mathcal{M}$).  Already we can see that the functor $Q_{\B}$ need be neither full nor faithful in general.  Indeed if $\mathbb{N}_0$ is the additive monoid of non-negative integers then $Q\mathbb{N}_0=\mathbb{Z}$, hence $Q_{\mathbb{N}_0}$ is not full.  If $\mathcal{E}=\{1,e\}$ is the monoid where $e^2=e$ then $Q\mathcal{E}$ is the trivial group so $Q_{\mathcal{E}}$ is not faithful.  Determining when $Q_{\mathcal{M}}$ is faithful is known as the ``group embeddability'' question, see \cite[1.10]{CP}.
\end{rem}

We end this Section by noting that $Q$ is a functor from  the category of small categories to the category of groupoids. Indeed let $F:\C\to\B$ be a functor and consider its composition $\C\to Q\B$ with $Q_\B$. By the universal property of the category of fractions of $\C$ there is a unique functor $QF:Q\C\to Q\B$ making the following square commute:
$$
\xymatrix{
\C \ar[d]_{F} \ar[r]^{Q_{\C}} & Q\C\ar[d]^{QF}\\
\B \ar[r]_{Q_{\B}}   &Q\mathcal{B}.}
$$

\section{Galois coverings and groupoids}\label{coverings}

Coverings of topological spaces (see for instance \cite{H}) are studied  by  J.P. May in \cite[Chapter 3]{May} by first considering coverings of groupoids. In order to study  coverings of categories we will follow a similar approach in this section. In particular we will prove that for a category $\B$ the functor $Q$ induces an equivalence between the category of Galois coverings of $\B$ and the category of Galois coverings of the groupoid $Q\B$.

\begin{defi}\cite[Chapter 13]{HI}, \cite[10.2]{brGroupoids}, \cite[4.1 and 4.17]{ta}.
Let $\B$ be a category and let $b$ be an object of $\B$.
The \textbf{source star} at $b$ is the set of morphims of $\B$ with source $b$. Similarly
the \textbf{target star} at $b$ is the set of morphisms of $\B$ with target $b$. The \textbf{star} at $b$ is the disjoint union of the source star at $b$ and the target star at $b$.

A functor $F:\mathcal{C}\to \mathcal{B}$ is a \textbf{covering} of $\mathcal{B}$ if $F$ is surjective on objects and if for each object $c$ in $\C$ the map induced by $F$ from the star at $c$ to the star at $Fc$ is bijective. Equivalently $F$ has to induce bijections between corresponding source and target stars. The covering $F$ is said to be \textbf{connected} if  $\C$ is connected.
\end{defi}

Note that if a covering $\C$ of $\B$ is connected, then $\B$ must also be connected.  That the functor $F$ is surjective on objects follows easily from the rest of the definition when $\B$ is connected and $\C$ is non-empty.

\begin{rem}
If $\B$ is a groupoid and $F: \C\to\B$ is a surjective on objects functor, it is sufficient to require that $F$ induces bijections on source (or target) stars in order to be a covering.  Note also that if $F$ is a covering and $\B$ is a groupoid then $\C$ is also a groupoid.
\end{rem}

Let $F$ be a covering.  Since $QF$ and $F$ coincide on objects, $QF$ is surjective on objects.

We provide first results concerning coverings which are analogous to those obtained in the linear case, see \cite{CRS} and \cite{le,le1,le2}.  If $F:\C\to \B$ and $G:\D\to \B$ are coverings, a \textbf{map of coverings} $F\to G$ is a functor $H:\C\to\D$ such that $GH=F$.

\begin{rem}
In case of linear categories over a field, a morphism of coverings needs to include in its data an automorphism $J$ of $\B$ in order to ensure that some coverings are isomorphic, see \cite{le2,le,CRS}. In the present context there is no need to do so since if $H: \C\to \D$ is an invertible functor and $J:\B\to\B$ is an automorphism satisfying $GH=JF$, then there is a unique invertible functor $H':\C\to\D$ such that $GH'=F$. Such an $H'$ can be constructed in this context (see below) essentially because any $F$-lifting of a morphism of $\B$ is still a morphism of $\C$ while in the linear situation an  $F$-lifting of a morphism of $\B$ is in general a non trivial sum of morphisms having same source and different targets (or same target and different sources). The construction of $H'$ is as follows: it coincides with $H$ on objects while for a morphism $f$ we consider $J^{-1}(F(f))$ and its $F$-lifting $f'$ having the same source as $f$, so that $F(f')= J^{-1}(F(f))$.  Then we define $H'$ by $H'(f)= H(f')$.
\end{rem}

The \textbf{automorphism group} $\aut F$ of a covering is the group of invertible maps of coverings from $F$ to $F$.  It acts on the {$F$-fibres} of objects of $\B$  in the natural way, and the action is free by similar arguments to those in \cite[Corollary 2.10]{CRS}.  Moreover if $\C$ is connected the action is transitive on every fibre if and only if it is transitive on any particular one - for the proof see \cite[Proposition 3.6]{CRS}.

\begin{defi}
A covering of categories $F:\C\to \B$ is a \textbf{Galois covering} if $\C$ is connected and if its automorphism group acts transitively on some fibre (or equivalently on any fibre).
\end{defi}

Galois coverings can be modeled using group actions and corresponding orbit categories as follows.

\begin{defi}
An action of a group $\Gamma$ on a category $\C$ is a group homomorphism from $\Gamma$ to the group of automorphisms of $\C$.
This provides an action of $\Gamma$ on objects and on morphisms.  For $\gamma\in \Gamma$ and $f\in {}_y\B_x$ we have that $\gamma f\in {}_{\gamma y}\B_{\gamma x}$.  Moreover $\delta(\gamma f)=(\delta\gamma)f$ for $\gamma,\delta\in \Gamma$, and $\gamma(gf)=(\gamma g)(\gamma f)$ whenever $g$ and $f$ are composable morphisms of $\C$.
\end{defi}

When a group $\Gamma$ acts freely on the objects of a category $\C$, we can form an \textbf{orbit category} $\C/\Gamma$. Its objects are the orbits of objects under $\Gamma$, while the morphisms from an orbit $\alpha$ to an orbit $\beta$ are the orbits of the action of $\Gamma$ on the disjoint union $$\coprod_{x\in\alpha, y\in\beta} {}_y\C_x.$$
Composition is well defined precisely because the action of $\Gamma$ on $\C_0$ is free. Moreover the projection functor $P:\C\to \C/\Gamma$ is a covering. Its automorphism group is $\Gamma$ and it acts transitively on the fibres of objects. If moreover $\C$ is connected then $P$ is a Galois covering. These facts are easy to prove, see also \cite[Proposition 3.4]{CRS} for the linear case.

\begin{thm} \cite[Theorem 3.7]{CRS}\label{galois are quotients}
Let $F:\C\to \B$ be a Galois covering and let $P:\C\to \C/\aut F$ be the Galois covering obtained using the free action of $\aut F$ on $\C$. There is a unique isomorphism of categories $I: \C/\aut F\to \B$ such that $IP=F$.
\end{thm}

We will use the previous canonical form for a Galois covering.

\begin{thm}\label{Qgalois}
Let $\C$ be a connected category together with a free action of a group $\Gamma$ and let $P: \C\to \C/\Gamma$ be the corresponding Galois covering. Then $\Gamma$ acts freely on $Q\C$ and the quotient category $Q\C/\Gamma$ is canonically isomorphic to $Q(C/\Gamma)$. Moreover through this canonical isomorphism the projection functor $Q\C\to (Q\C)/\Gamma$ corresponds to $QP$:

$$
\xymatrix{
\C \ar[d]_{P} \ar[rr]^{Q_{\C}}  & & Q\C\ar[d]^{QP}\\
\C/\Gamma \ar[rr]_{Q_{\C/\Gamma}}  & & Q(\C/\Gamma)=(Q\C)/\Gamma }
$$
Consequently $QP$ is a Galois covering.
\end{thm}

\begin{proof}
First we check that $\Gamma$ acts on $Q\C$. Recall that $\C_1'$ is a copy of $\C_1$ with the direction of the arrows reversed.  Hence $\Gamma$ acts on the directed graph $\C_1'$ as it does on $\C_1$. This gives an action of $\Gamma$ on the $\mathcal{C}$-free category $\mathcal{F}_{\mathcal{C}}\mathcal{C}_1'$ and it is clear that the ``inverting'' relations required to obtain $Q\C$ from $\mathcal{F}_{\mathcal{C}}\mathcal{C}_1'$ are stable under the action of $\Gamma$.  It follows that the set of all ideal relations containing these relations is closed under the action of $\Gamma$.  Hence the intersection of these ideal relations (that is, the ideal relation generated by the inverting relations) is $\Gamma$-stable. Consequently $F$ acts on $Q\C$, and since the actions of $\Gamma$ on $\C$ and on $Q\C$ agree on objects, the action on $Q\C$ is free on objects.

We will prove next that the categories $Q\C / \Gamma$ and $Q (\C/\Gamma)$ are isomorphic. By direct inspection there is a well defined and canonical functor $\C/\Gamma\to (Q\C)/\Gamma$. The category $(Q\C)/\Gamma$  is a groupoid since quotients of groupoids by free actions are groupoids.  By the universal property we infer a unique functor $Q(\C/\Gamma)\to (Q\C)/\Gamma$.  Conversely note that $QP$ is constant on orbit objects and morphisms.  Clearly the quotient category by a free action enjoys the universal property implying that this determines the existence of a unique functor $(Q\C)/\Gamma \to Q(\C/\Gamma)$.  The two functors described are maps of coverings and are mutually inverse. \qed
\end{proof}

Given a covering of $Q\mathcal{B}$ we will construct a corresponding covering of $\mathcal{B}$ as a pullback.  For this we recall the following easy result from \cite[Proposition 29]{HI}. Note however that the definition of covering in \cite{HI} is not the same as ours - P.J. Higgins does not require surjectivity on objects and only demands that the functor be bijective on source stars.  Nevertheless the proof of the following goes through.

\begin{pro}\cite[Proposition 29]{HI}
Consider a pullback of categories:

$$
\xymatrix{
\mathcal{B}\times_{\mathcal{D}}\mathcal{G} \ar[d]_{F} \ar[r] & \mathcal{G}\ar[d]^{G}\\
\mathcal{B} \ar[r]_{\theta} &\mathcal{D}}
$$
If $G$ is a covering then $F$ is a covering.\qed
\end{pro}

The following is also clear:
\begin{pro}
In the situation above assume that $G$ is a Galois covering, that $\B$ and $\D$ have the same set of objects and that $\theta$ is the identity on objects. Then $F$ is a Galois covering and $\aut F=\aut G$.\qed
\end{pro}

The category of Galois coverings of a category $\B$ with morphisms maps of coverings is denoted $\galcov_\B$.  We will also consider another category with fewer morphisms, namely the category $\galcov (\B,b_0)$ of \textbf{pointed Galois coverings} with respect to a chosen base-object $b_0$ of $\B$. Objects of $\galcov (\B,b_0)$ are Galois coverings $(F,c):\C \to\B$ where $c$ is an object in the fibre of $b_0$ under $F$.  A morphism $H$ from $(F,c)$ to $(G,d)$ is a functor $H$ such that $GH=F$ and $Hc=d$.

\begin{thm}\label{galcovB and galcovQB}
Let $\B$ be a connected category and let $Q\B$ be its category of fractions. The categories $\galcov_\B$ and $\galcov_{Q\B}$ are equivalent, as are the categories $\galcov(\B,b_0)$ and $\galcov(Q\B,b_0)$ for a fixed base object $b_0$.
\end{thm}

\begin{proof}
We have already noted that applying the functor $Q$ to a Galois covering gives a Galois covering.  Conversely the pullback considered above provides a Galois covering of $\B$ from a Galois covering of $Q\B$. It is straightforward to verify that these functors are mutually pseudo-inverse. \qed
\end{proof}

\begin{rem}
A more general result is true.  Let $\cat{Cov}_{\B}$ be the category of all connected coverings of $\B$ with morphisms maps of coverings.  Then the functor $\textbf{Cov}_{\B}\to \textbf{Cov}_{Q\B}$ sending $F\to QF$ is an equivalence of categories.  Since our main interest in this paper is Galois coverings, we do not prove this here.
\end{rem}

A Galois covering is \textbf{universal} if it covers uniquely any other Galois covering. J.P. May proved in \cite[Chap. 3 Section 6]{May} that connected groupoids admit universal covers.  We infer the following

\begin{cor} (see \cite{OJ,ta})
Let $\B$ be a small connected category with chosen base-object $b_0$.  The category $\galcov(\B,b_0)$ admits a universal covering of $\B$.
\end{cor}

We will also give a constructive proof of this fact in Theorem \ref{universal}.

Next we will prove that the category of pointed Galois coverings is \textbf{sharp}, meaning that each morphism set has at most one element. Note that a sharp category is precisely a preorder between its objects \emph{i.e.} any skeleton of the category is a partially ordered set. The following two results are originated in the work by P. Lemeur in the linear context, see \cite{le,le1,le2}.

\begin{pro}\label{equal}
Let $F:\C\to \B$ and $G:\D\to \B$ be connected coverings of a category $\B$. Let $H$ and $H'$ be two morphisms from $F$ to $G$ which coincide on some object. Then $H=H'$. In particular the category of pointed connected coverings is sharp.
\end{pro}

\begin{proof}
Let $c$ be an object of $\C$ such that $Hc=H'c$ and let $f$ be a morphism having source or target object $c$, say $s(f)=c$. We assert that $H(f)=H'(f)$.  Note that $GH(f)$ and $GH'(f)$ are equal since both coincide with $F(f)$. Hence  $H(f)$ and $H(f')$ are morphisms with the same source in $\D$ and having the same image under $G$.  Bijectivity of $G$ on stars now ensures that $H(f)=H'(f)$.  As a consequence $H(t(f))=H'(t(f))$. Since $\C$ is connected the result follows.\qed
\end{proof}

\begin{cor}\label{lambda}
Let $F:\C\to \B$ and $G:\D\to \B$ be Galois coverings of a category $\B$ with automorphism groups $\Gamma$ and $\Gamma'$ respectively. Let $H$ be a morphism from $F$ to $G$. Then
\begin{itemize}
\item
There is a unique surjective group homomorphism $\lambda_H:\Gamma\to \Gamma'$ such that $H\gamma=\lambda_H(\gamma)H$ for all $\gamma$ in $\Gamma$.
\item
$H$ is a Galois covering with automorphism group $\ker \lambda_H$.
\end{itemize}
\end{cor}

\begin{proof}
Note that if an object $d$ is in the image of $H$ (say $d=Hc$) then the star at $d$ is also in the image of $H$.  To see this consider $\delta$ in the star at $d$ and its image $G(\delta)$ in the star at $GHc=Fc$.  Since $F$ is surjective on stars there is $\varepsilon$ in the star at $c$ mapping to $G(\delta)$ under $H$.  Thus $GH(\varepsilon)=G(\delta)$ so $H(\varepsilon)=\delta$ since $G$ is injective on stars. Since $\D$ is connected, we infer  that $H$ is surjective on objects.

The composition $H\gamma$ is a morphism from $F$ to $G$.  The previous proposition ensures that the map $\gamma\mapsto H\gamma$ is injective. Moreover any morphism $H'$ from $F$ to $G$ is of this form since there is a $\gamma\in\Gamma$ such that $H\gamma$ and $H'$ coincide on some object, so they are equal.  Similarly $\tau H$ is also a morphism for any $\tau\in\Gamma'$ and any morphism can be written uniquely in this form.  We thus have a well defined and unique map $\lambda_H:\Gamma\to\Gamma'$ such that $H\gamma=\lambda_H(\gamma)H$.  One can verify that $\lambda_H$ is a group morphism by direct inspection.  It is surjective by the above results.

In order to prove that $H$ is a covering let $g$ be a fixed morphism of $\D$ and let $c$ be a fixed object in the $H$-fibre of the source of $g$. We consider $G(g)$ which has a (unique) $F$-pre-image $f$ with source $c$, hence $F(f)=G(g)$. Now $GH(f)=F(f)$ so that $H(f)$ and $g$ have the same image in $\B$ and the same source object. Since $G$ is a covering we have $H(f)=g$. Moreover $f$  is unique with this property since otherwise $G(g)$ would have two $F$-pre-images with a fixed source, but $F$ is a covering.

Observe that $\gamma \in \ker \lambda_H$ if and only if $H\gamma=H$.  In order to prove that $H$ is a Galois covering, let $c$ and $c'$ be objects in the same $H$-fibre, so in the same $GH$-fibre - that is, the same $F$-fibre. Since $F$ is a Galois covering there exists $\gamma\in\Gamma$ such that $\gamma c=c'$. Then $H\gamma$ and $H$ are both morphisms from $F$ to $G$ which coincide on $c$ so by the previous proposition $H\gamma=H$. \qed
\end{proof}

\section{Gradings and smash products}\label{gradings}

We will define next a Galois covering of $\B$ constructed from a grading of $\B$.  Let $\G$ be a groupoid.  A \textbf{groupoid grading} (often just called a grading) of $\B$ is a functor $X:\B\to \G$. By the universal property of the category of fractions, $X$ factors uniquely through $Q\B$ providing an associated grading $QX:Q\B\to\G$ of $Q\mathcal{B}$.

\begin{defi}
A  grading $X:\B\to\G$ is \textbf{effective} if $\G$ is connected, $X$ is bijective on objects and if $QX$ is full.
\end{defi}

\noindent We will show that effective gradings of $\B$ correspond to Galois coverings of $\B$. In particular the effective condition on gradings will correspond to the fact that a Galois covering is connected.

 We briefly recall some properties of groupoids, see for instance \cite{brGroupoids}.  If $\G$ is a groupoid, then a \textbf{subgroupoid} is a subcategory which is itself a groupoid.  A subgroupoid $\N$ of $\G$ is \textbf{normal} in $\G$ if $\N_0=\G_0$ and if for any objects $x,y\in \G_0$ and any $\gamma\in {}_y\N_x$  we have $\gamma^{-1} \left({}_y\N_y\right) \gamma\subseteq {}_x\N_x$.  A groupoid is said to be \textbf{totally disconnected} if there are no morphisms between distinct objects of the groupoid.  Let $\N$ be a totally disconnected normal subgroupoid of the groupoid $\G$.  The \textbf{quotient groupoid} $\G/\N$ is the quotient of $\G$ by the following ideal relation: two morphisms $f$ and $g$ from $x$ to $y$ are equivalent if and only if there exist $\beta\in {}_y\N_y$ and $\alpha\in{}_x\N_x$ such that $g=\beta f\alpha$.

Let $F:\G\to \mathcal{H}$ be a functor between groupoids that is injective on objects.  Then $\Im(F)$ is a subgroupoid of $\mathcal{H}$.  The \textbf{kernel} $\ker (F)$ of $F$ is the subcategory of $\G$ with objects $\G_0$ and morphisms $\{\gamma\in\G_1 \ \mid \  F(\gamma)=\id\}$.  Since $F$ is injective on objects, $\ker (F)$ is a totally disconnected normal subgroupoid of $\G$.  What is more, the expected isomorphism theorem (\cite[8.3.2]{brGroupoids}) holds: $\G/\ker(F)\iso \im(F)$.

\medskip

Given two gradings $X:\B\to \G$ and $Y:\B\to \mathcal{H}$ a \textbf{map of gradings} from $X$ to $Y$ is a functor $Z:\G\to \mathcal{H}$ such that $ZX=Y$.  The category $\cat{EffGrad}_{\B}$ has objects effective gradings of $\B$ and morphisms as above. If $X: \B\to \G$ is an effective grading then $\G\iso Q\B/\ker(QX)$.  Moreover the grading $QX$ of $Q\B$ is isomorphic to the projection functor $Q\B \to Q\B/\ker (QX)$.

\medskip

The following results are immediate

\begin{lem}\label{sharp gradings}\label{morphisms of gradings}
Let $X:\B\to Q\B/\M$ and $Y:\B\to Q\B/\N$ be effective gradings of $\B$.  There is a map of gradings $Z:X\to Y$ if and only if $\M$ is a subgroupoid of $\N$, in which case $Z$ is the projection functor and hence unique.  In particular, the category of effective gradings of a connected category is sharp.\qed
\end{lem}

\begin{pro}\label{universal grading}
The grading of $\B$ given by its category of fractions $Q_\B:\B\to Q\B$ is effective and universal, in the sense that any other effective grading is uniquely a quotient of it.\qed
 \end{pro}

Given a linear grading of a category over a field, a categorical linear smash product has been considered as a generalization of the smash product of a graded algebra by a group, see \cite{CM}. Next we will consider an analogous smash product category associated to a groupoid grading.

\begin{defi}
Let $X:\B\to \G$ be a grading of a  category $\B$ and let $x_0$ be an object of $\G$.  The objects of the \textbf{smash product} category $\B\#_{x_0}X$ are pairs $(b,\gamma)$ where $b\in\B_0$ and $\gamma\in {}_{x_0}\G_{X\!b}$.  A morphism from $(b,\gamma)$ to $(c,\delta)$ is a morphism $f\in {}_c\B_b$ such that $\delta X(f)=\gamma$ or equivalently a morphism of $\B$ which has $X$-degree $\delta^{-1}\gamma$, where the degree of a morphism is its image by $X$.
\end{defi}

The smash product is an example of a ``comma category'' and  S. Mac Lane refers to $\B\#_{x_0}X$ as the category of ``objects $X$-over $x_0$'' (see \cite[II.6]{MacLane}).

Note that if $\B$ is a groupoid, the smash product is also a groupoid.  We next show that if $X$ is an effective grading of $\mathcal{B}$ then $\mathcal{B}\#_{x_0}X$ is a Galois covering of $\mathcal{B}$.

\begin{lem}
Let $X:\mathcal{B}\to \mathcal{G}$ be a grading that is bijective on objects and let $Y:\mathcal{B}\to \im (QX)$ be the corestriction of $X$ to the image of $QX$ (so that $Y$ is effective, \emph{i.e.} $QY$ is full).
Let $\mathcal{H}$ be the category
$$\coprod_{\alpha\in A}{}_\alpha\!\left(\mathcal{B}\#_{x_0}Y\right)$$
where $A$ is a set of left-coset representatives of ${}_{x_0}\im (QX)_{x_0}$ in ${}_{x_0}\mathcal{G}_{x_0}$ and ${}_\alpha\!\left(\mathcal{B}\#_{x_0}Y\right)$ is a copy of $\mathcal{B}\#_{x_0}Y$ .  Then $\mathcal{H}$ is isomorphic to $\mathcal{B}\#_{x_0}X$.
\end{lem}

\begin{proof}
Denote by ${}_{\alpha}(b,\gamma)$ the object $(b,\gamma)$ of the copy of $\mathcal{B}\#_{x_0}Y$ labeled by $\alpha\in A$.  Define a functor $\eta:\mathcal{H}\to \mathcal{B}\#_{x_0}X$ as follows.  On objects $\eta$ sends ${}_{\alpha}(b,\gamma)$ to $(b,\alpha\gamma)$.  On morphisms, $\delta$ maps to $\delta$.  This is a well-defined functor (not independent of choice of $A$) and is bijective on objects.  It is clearly faithful, so we need only check it is full. First note that if $\alpha\neq\alpha'$ there are no morphisms in $\mathcal{B}\#_{x_0}X$ from $(b,\alpha\gamma)$ to $(c,\alpha'\gamma')$. Indeed, if $\delta$ is such a morphism, then $\alpha'^{-1}\alpha=\gamma'X(\delta)\gamma^{-1}\in \im (QX)$ so that $\alpha=\alpha'$.  Consider now $\delta\in {}_{(c,\alpha\gamma')}\left( \mathcal{B}\#_{x_0}X   \right)_{(b,\alpha\gamma)}$.  Then $X(\delta)=\gamma'^{-1}\gamma$ so that $\delta$ is the image under $\eta$ of $\delta$ as a morphism in ${}_\alpha\!\left(\mathcal{B}\#_{x_0}Y\right)$.\qed
\end{proof}

\begin{lem}
Let $X:\mathcal{B}\to \mathcal{G}$ be an effective grading of $\mathcal{B}$.  Then $\mathcal{B}\#_{x_0}X$ is connected.
\end{lem}

\begin{proof}
Fix two objects $(b,\alpha)$ and $(c,\gamma)$ of $\mathcal{B}\#_{x_0}X$.  Since $QX$ is full there is a morphism $\delta$ of $Q\mathcal{B}$ mapping to $\gamma^{-1}\alpha$ under $QX$.  We know that  $\delta$ corresponds to a walk $(\delta_n,\hdots,\delta_1)$ from $b$ to $c$ in $\mathcal{B}$.  It is easy to check that we can choose objects in $\mathcal{B}\#_{x_0}X$ so that $(\delta_n,\hdots,\delta_1)$ can be considered as a walk in $\mathcal{B}\#_{x_0}X$ from $(b,\alpha)$ to $(c,\gamma)$ as needed -- for example, if $\delta_1$ were a ``reversed'' morphism of $\mathcal{B}$ then the beginning of the walk would be given by $\delta_1$ viewed as a morphism from  $\left(s(\delta_1),\alpha X(\delta_1)\right)$ to  $(b,\alpha)$.\qed
\end{proof}

From the previous two lemmas the following is immediate:

\begin{pro}
Let $X:\mathcal{B}\to \mathcal{G}$ be a grading of $\mathcal{B}$ that is bijective on objects.  Then $\mathcal{B}\#_{x_0}X$ is connected if and only if $X$ is effective.\qed
\end{pro}

We show next that the smash product construction provides Galois coverings. Consider the canonical functor $F_{X}: \B\#_{x_0}X\to\B$ given by $F_X(b,\gamma)=b$ and $F_X({}_{(b',\gamma')}\delta_{(b,\gamma)})=\delta$.

The group ${}_{x_0}\G_{x_0}$ acts on the smash product as follows: for $\alpha\in {}_{x_0}\G_{x_0}$ we have $\alpha(b,\gamma) = (b,\alpha\gamma)$ and $\alpha \delta=\delta$. Note that this action is free on the set of objects; the next result is clear.

\begin{pro}
Let $\B$ be a category and let $X:\B\to\G$ be a grading of $\B$. The quotient category $\left(\B\#_{x_0}X\right)/{}_{x_0}\G_{x_0}$ is canonically isomorphic to $\B$. If $X$ is effective then $F_X:\B\#_{x_0}X\to \B$ defined as above is a Galois covering of $\B$ with $\aut \left(F_X\right)={}_{x_0}\G_{x_0}$.\qed
\end{pro}

In order to describe morphisms between smash products, we first extend the surjective group morphism obtained in Proposition \ref{lambda} to the  grou\-poid. Let $X:\B\to\G$ and $Y : \B\to\mathcal{H} $ be effective gradings of $\B$ and let $H$ be a morphism of Galois coverings from $F_X$ to $F_Y$. We have proven the existence of a  surjective group homomorphism $\lambda_H : {}_{x_0}\G_{x_0}\to {}_{x_0}\mathcal{H}_{x_0}$. Let $\alpha_x$ be a chosen morphism from $x$ to $x_0$ for each object $x$ of $\G$, with $\alpha_{x_0}=1$. Then $\lambda_H$ extends to a functor $\lambda_H'$ from $\G$ to ${}_{x_0}\mathcal{H}_{x_0}$ by the formula $\lambda'_H (\epsilon) = \lambda_H \left( \alpha_{t(\epsilon)}\ \epsilon\  \alpha_{s(\epsilon)}^{-1}\right)$.

\begin{pro}\label{morphisms smashes}
In the situation above, for $b\in \B$ define a map $H_b:{}_{x_0}\G_{_{Xb}}\to {}_{x_0}\mathcal{H}_{_{Yb}}$ by setting $H(b,\gamma)=(b,H_b(\gamma))$ for $(b,\gamma)$ in $\B\#_{x_0}X$.  Then $H_b(\gamma)=\lambda'_H(\gamma)H_b(\alpha_{_{Xb}})$.  Moreover,
the image under $H$ of a morphism $f$ in $\B\#_{x_0}X$ from $(b,\gamma)$ to $(c,\delta)$

is $f$ itself and
$$Y(f)= H_c\left(\alpha_{_{Xc}}\right)^{-1} \ \lambda'\left(X(f)\right) \ H_b\left(\alpha_{_{Xb}}\right).$$
\end{pro}

\begin{proof}
By definition of $\lambda_H$ we have that
$$(b, H_b(\gamma))= H(b, \gamma) = H(b, \gamma\alpha_{_{Xb}}^{-1}\alpha_{_{Xb}}) = \lambda_H(\gamma\alpha_{_{Xb}}^{-1})H(b,\alpha_{_{Xb}})=$$
$$\lambda'_H(\gamma)\left(b,H_b\left(\alpha_{_{Xb}}\right)\right).$$
The formula for $Y(f)$ can also be seen by direct computation.\qed
\end{proof}

Finally we note that the smash product defined above is an instance of a pullback. Consider the \textbf{category over} $x_0$ (also called the slice category or category of objects over $x_0$), denoted ${}_{x_0}\G$: its objects are those morphisms $\alpha$ of $\G$ with target $x_0$.  A morphism from $\alpha$ to $\gamma$ is $\delta\in {}_{s(\alpha)}\G_{s(\gamma)}$ such that $\gamma\delta=\alpha$.  Since $\G$ is a groupoid, there is exactly one morphism $\gamma^{-1}\alpha$ from $\alpha$ to $\gamma$ in the star of $x_0$, so that ${}_{x_0}\G$ is a connected trivial groupoid.   Clearly ${}_{x_0}\G = \G\#_{x_0}\id$ and there is a Galois covering ${}_{x_0}\G\to\G$. The proof of the following is straightforward.

\begin{pro}
Let $\B$ be a category and $X:\B\to\G$ be a grading. The categories $\B \times_\G \left({}_{x_0}\G\right)$ and $\B\#_{x_0}X$ are isomorphic.
\end{pro}

\section{Effective gradings and pointed Galois coverings}\label{effective}

In this section we will prove that the category of effective gradings of a category $\B$ and the category of pointed Galois coverings of $\B$ are equivalent. First we show how to obtain an effective grading of $\B$ from a Galois covering of $\B$ provided with a choice of an object in each fibre.

Let $F:\C\to \B$ be a Galois covering of a category $\B$.  In other words, according to Theorem \ref{galois are quotients}, we consider a group $\Gamma$ acting freely on $\C$ and let $\B$ be the orbit category $\C/\Gamma$.  The orbit of an object $x\in \C_0$ is denoted $[x]$ and likewise for morphisms.

In order to define a grading we consider an \textbf{object-section} of the Galois covering as follows: for each object orbit $\omega$ we choose a representative $\omega_0\in\omega$, or equivalently an object $\omega_0$ in the $F$-fibre of $\omega$.  Associated to this object-section there is a deviation map $d:\C_0\to \Gamma$ determined by
$$x=d(x)[x]_0$$
which is well defined since the action of $\Gamma$ on objects is free.  Moreover $d$ is a map of $\Gamma$-sets, so that $d(\gamma x)=\gamma d(x)$ for each $\gamma\in \Gamma$ and $x\in \C_0$. Any other object-section is of the form $[x]_1=u_{[x]} [x]_0$ for a unique $u_{[x]}\in \Gamma$; the corresponding deviation $d_1$ is given by $d_1(x)=d(x)u_{[x]}$.

\begin{defi}\label{associated grading}
Let $\C$ be a connected category and let $\Gamma$ be a group acting freely on $\C$. Let $\G$ be the groupoid whose objects are $\left(\C/\Gamma\right)_0$ and having a copy of $\Gamma$ as the set of morphisms between any two objects. Consider an object-section of the Galois covering $\C\to\C/\Gamma$ and let $d$ be the corresponding deviation. The \textbf{associated grading} $X:\C/\Gamma\to\G$ is the identity on objects (\emph{i.e}. on $\Gamma$-orbits).  For an orbit morphism $[f]$ from $[x]$ to $[y]$ let $f'$ be a representative and let $x'$ and $y'$ be its source and target objects in $\C$.  We define $X([f])=d(y')^{-1}d(x')$.
\end{defi}
Note that $X$ is indeed a functor and that $X([f])$ does not depend on the choice  of the morphism representative $f'$.

\begin{rem}
For another object-section of the covering with deviation $d'(x)=d(x)u_{[x]}$ the corresponding grading $X'$ is given by  $$X'\left({}_y\!f_x\right) = u_{[y]}^{-1}X([f])u_{[x]}.$$
\end{rem}

\begin{lem}
The above grading is effective.
\end{lem}

\begin{proof}
We need to show that $QX$ is full.  It is easy to check that $QX$ is the grading associated to the Galois covering $QF:Q\C\to Q(\C/\Gamma)=Q\C/\Gamma$, so we reduce to the case where $\C$ is a groupoid.

Let $\gamma$ be a morphism from $[x]$ to $[y]$ in $\G$. Since $\C$ is a connected groupoid, there is a morphism $f$ from $[x]_0$ to $\gamma^{-1}[y]_0$ in $\C$. Clearly $X([f])=\gamma$. \qed
\end{proof}

\begin{rem}
An effective grading has only one endomorphism, while a Galois covering can have many.  Consequently $\galcov_\B$ and $\cat{EffGrad}_\B$ are not equivalent categories. On the other hand, the category of \emph{pointed} Galois coverings is sharp like the category of effective gradings .
\end{rem}

\begin{thm}\label{EffGrad equiv GalCov}
Let $\B$ be a connected category and let $b_0$ be an object of $\B$.  The categories $\cat{EffGrad}_\B$ and $\galcov(\B,b_0)$ are equivalent.
\end{thm}

\begin{proof}
We define a functor $\#_{b_0}:\cat{EffGrad}_\B\to \galcov(\B,b_0)$ as follows.  Given an effective grading $X$ of $\B$, the associated Galois covering is $\B\#_{Xb_0}X$.  We provide it with the object in the fibre of $b_0$ given by the identity endomorphism $1_{Xb_0}$ of $Xb_0$.  A map $Z$ of gradings from $X$ to $Y$ clearly induces a map of coverings $\#_{b_0}(Z):\B\#_{Xb_0}X\to\B\#_{Yb_0}Y$.  We will show that $\#_{b_0}$ is essentially surjective and fully faithful.

Let $F: \C \to \C/\Gamma$ be a Galois covering with automorphism group $\Gamma$ acting freely on $\C$.  As in Definition \ref{associated grading} choose an object-section and let $X:\C/\Gamma\to \G$ be the associated grading of $\C/\Gamma$ with corresponding deviation map $d$,  and choose a base-object $[q]$ of $\G$.

We assert that the smash product $\left(\C/\Gamma\right) \#_{[q]} X$ is isomorphic as a Galois covering to $\C$.  Consider an object $([x],\gamma)$ of the smash product and set $\gamma[x]_0$ to be the corresponding object of $\C$.  Consider a morphism $[f]$ of the smash product with source $([x],\gamma)$ and target $([y],\delta)$ (so that $X([f])=\delta^{-1}\gamma$).  We would like to associate to $[f]$ the representative $f'$ having source $\gamma [x]_0$.  To do this we must check that the target of $f'$ is $\delta[y]_0$.  Write the target of $f'$ as $\beta[y]_0$ for some $\beta\in \Gamma$.  Then
$$\delta^{-1}\gamma=X([f])=X([f'])=\beta^{-1}\gamma,$$
implying that $\beta=\delta$ as required.   It is now straightforward to check that the maps $([x],\gamma)\mapsto \gamma[x]_0$ and $[f]\mapsto f'$ define an isomorphism of pointed Galois coverings.

That $\#_{b_0}$ is faithful is immediate, since we know by Lemma \ref{sharp gradings} that $\cat{EffGrad}_\B$ is sharp.

It remains to show that $\#_{b_0}$ is full.  Let $X:\B\to \G$ and $Y:\B\to \mathcal{H}$ be effective gradings of $\B$, let $H:\B\#_{Xb_0}X\to \B\#_{Yb_0}Y$ be a map of coverings and let $\lambda_H$ be the corresponding group homomorphism given in Corollary \ref{lambda}.  Employing the notation from Proposition \ref{morphisms smashes} we have that $H(b,\gamma) = \left(b, \lambda'_H(\gamma) H_b\left(\alpha_{_{Xb}}\right)\right)$.  Since $\galcov(\B,b_0)$ is sharp it suffices to show that there is a map of gradings $X\to Y$, and this clearly follows if there is a map of gradings $QX\to QY$.  By Lemma \ref{sharp gradings} this is equivalent to showing that $\ker(QX)\subseteq \ker(QY)$.  To this end consider the map of coverings $Q\B\#_{Xb_0}QX \to Q\B\#_{Yb_0}QY$ given by $\lambda_H$ as follows: on objects it coincides with $H$, while on a morphism $\phi$ of $Q\B$ from $(b,\gamma)$ to $(c,\delta)$ (so that $QX(\phi)=\delta^{-1}\gamma$) its value is $\phi$ considered as a morphism from $(b,\lambda'(\gamma)H_b(\alpha_{_{Xb}}))$ to $(c,\lambda'(\delta)H_c(\alpha_{_{Xc}}))$.
This makes sense if and only if 
$$(QY)(\phi) = H_c\left(\alpha_{_{Xc}}\right)^{-1} \ \lambda_H'\left((QX)(\phi)\right) \ H_b\left(\alpha_{_{Xb}}\right).$$

By Proposition \ref{morphisms smashes} this formula is valid when restricted to morphisms in $\B$.

Consider now $Y'$ defined this way on $Q\B$, namely $Y'$ is the identity on objects and for $\phi$ a morphism in $Q\B$ from $b$ to $c$ we set

$$(Y')(\phi) = H_c\left(\alpha_{_{Xc}}\right)^{-1} \ \lambda_H'\left((QX)(\phi)\right) \ H_b\left(\alpha_{_{Xb}}\right).$$

This defines a functor $Q\B\to \mathcal{H}$ which coincides with $Y$ on $\B$ (i.e. $Y'Q_{\B}= Y$).  Since $QY$ is the unique functor satisfying $(QY)Q_{\B} = Y$, it follows that $Y'=QY$.

Now if $\beta$ is in $\ker(QX)$ then $\beta$ provides an endomorphism of the object $(b_0, 1_{Xb_0})$, hence its image  $QH(\beta)$ is an endomorphism of the object $(b_0, 1_{Yb_0})$ which means that it belongs to the kernel of $QY$. Note that since $QH$ is a morphism of coverings we have that $QH(\beta)=\beta$.
\qed
\end{proof}

A pointed Galois covering of $\B$ is \textbf{universal} if it covers uniquely every pointed Galois covering of $\B$.

\begin{thm}\label{universal}
Let $\B$ be a connected category, let $b_0$ be a chosen base object of $\B$ and let $Q_\B:\B\to Q\B$ be the canonical functor to the category of fractions.  Then $\B\#_{b_0}Q_{\B}$ with the identity of $b_0$ as given object in the fibre of $b_0$ is the universal Galois covering of $\B$.
\end{thm}

\begin{proof}
We already know that between two pointed Galois coverings there is at most one morphism. Let $F:\C\to \B$ be a Galois covering of $\B$, corresponding to an effective grading $X:\B\to \G$.  By the universal property of $Q\B$, we have a map of gradings $QX:Q\B\to \G$.  By Theorem \ref{EffGrad equiv GalCov} this corresponds to a map of coverings $\B\#_{b_0}Q_{\B}\to \C$.  Finally, note that a map of coverings is itself a covering (see proof of Corollary \ref{lambda}).
\qed
\end{proof}

We end this paper with an example which describes the universal cover of the category $\K_E$ defined as follows. There are two objects $x_0$ and $x$, both having only identities as endomorphims.  There are no morphisms from $x_0$ to $x$ while ${}_{x_0}\left(\K_E\right)_{x}$ is an arbitrary non-empty set $E$.  There are no compositions to specify. In case $E$ has two elements the linearisation of $\K_E$ over a field is often called the Kronecker category.

Let $\alpha$ be a chosen element of $E$ and let  $E_\alpha=E\setminus \{\alpha\}$. The following result is  well-known and easy to prove:

\begin{lem}
Let $\F(E_\alpha)$ be the free group on the set $E_\alpha$. The category of fractions $Q=Q\K_E$ is the groupoid with two objects corresponding to the group $\F(E_\alpha)$.

More precisely the group at $x$ is the free group on the set of closed walks $\left(\beta', \alpha\right)_{\beta \in E_\alpha}$ while at $x_0$ the group is free on $\left(\alpha, \beta'\right)_{\beta \in E_\alpha}$. As with any groupoid the morphism set  from $x$ to $x_0$ is ${}_{x_0}Q_{x_0} \alpha$. This provides an identification between ${}_{x_0}Q_{x_0}$ and ${}_{x_0}Q_{x}$.
\end{lem}

Recall that the \textbf{Cayley graph} of a group $G$ with respect to a given set of generators $S$ has vertices the elements of the group and an arrow from $g$ to $gs$ for every $g\in G, s\in S$.

\begin{defi}
The \textbf{double} of a graph $C$ has vertices  $\{x,x_0\} \times C_0 $ (we think of two copies of $C$ labeled by $x$ and $x_0$ sitting side-by-side).  For each arrow $a$ in $C_1$ there is an arrow in the double graph from $(x,s(a))$ to $(x_0,t(a))$ (we think of arrows as moving left to right). Moreover there are arrows from $(x,y)$ to $(x_0,y)$ for every object $y$ (in our picture, these arrows are horizontal).
\end{defi}

\begin{pro}
The universal cover of $\K_E$ (that is, the smash product  $\K_E\#_{x_0}Q$) is the category given by the double of the Cayley graph of the free group $\F_{E_\alpha}$ at $x_0$ with respect to set of generators given above.
\end{pro}

\begin{rem}
Note that there are no concatenated arrows in the double of a graph.  Consequently we need not specify a composition law in order to define the category, apart from trivial compositions with identities.
\end{rem}

\begin{proof}
The objects of the smash product are of two kinds : $(x_0, \gamma)$ with $\gamma\in {}_{x_0}Q_{x_0}$ and $(x,\delta)$ with $\delta\in {}_{x_0}Q_{x}$. Each morphism in $K_E$ (namely each $\beta\in E$) provides a morphism from $(x,\delta)$ to $(x_0, \delta\beta^{-1})$. Through the identification above between ${}_{x_0}Q_{x_0}$ and ${}_{x_0}Q_{x}$ we obtain precisely the double of the Cayley graph.\qed
\end{proof}

In case the set $E$ has two elements the category is

$$
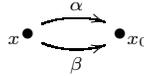
\xymatrix{
_{x}\bullet  \ar@/^/[r]^{\alpha}\ar@/_/[r]_{\beta} & \bullet_{x_0}}
$$

The free group involved has one generator $\alpha\beta'$ and the double of its Cayley graph is as follows - note that as above each arrow provides right multiplication by the inverse of its label on the second component:

 $$
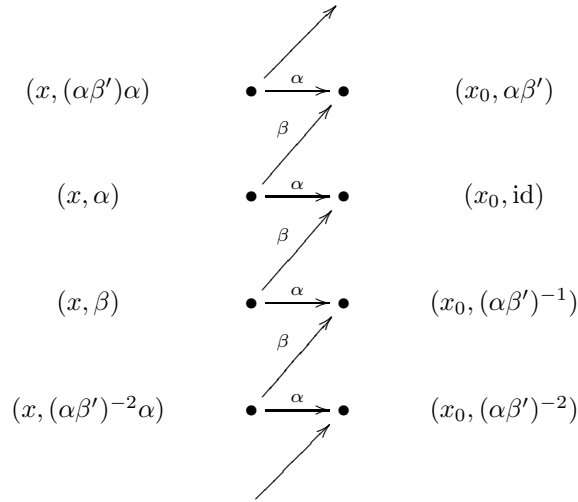
\xymatrix{
&&& \\
(x,(\alpha\beta')\alpha)&\bullet  \ar[r]^{\alpha}\ar[ur] & \bullet &(x_0,\alpha\beta')  \\
(x,\alpha)&\bullet  \ar[r]^{\alpha}\ar[ur]^{\beta} & \bullet &(x_0,\id)  \\
(x,\beta)&\bullet  \ar[r]^{\alpha}\ar[ur]^{\beta} & \bullet &(x_0,(\alpha\beta')^{-1})  \\
(x,(\alpha\beta')^{-2}\alpha)&\bullet  \ar[r]^{\alpha}\ar[ur]^{\beta} & \bullet &(x_0,(\alpha\beta')^{-2})  \\
&\ar[ur]&&}
$$


\footnotesize \noindent C.C.: \\Institut de math\'{e}matiques et de mod\'{e}lisation de Montpellier I3M,\\ UMR 5149\\ Universit\'{e}  Montpellier 2, \\F-34095 Montpellier cedex 5, France.\\ {\tt Claude.Cibils@math.univ-montp2.fr}
\smallskip

\noindent J.MacQ.: \\ Department of mathematics, \\ University Walk,\\ University of Bristol,\\Bristol, England, BS81TW. \\ {\tt John.MacQuarrie@bristol.ac.uk}


\begin{thebibliography}{99}




\bibitem{bo} Bongartz, K.; A criterion for finite representation type. Math. Ann. \textbf{ 269} (1984), no. 1, 1--12.

\bibitem{boga} Bongartz, K.; Gabriel, P. Covering spaces in representation-theory, Invent. Math. \textbf{ 65} (1981/82) 331--378.




\bibitem{brGroupoids} Brown, R. Topology and Groupoids. Booksurge LLC, S. Carolina, 2006.




\bibitem{CM} Cibils, C.; Marcos, E. Skew category, Galois covering and smash product of a category over a ring. Proc. Amer. Math. Soc.  \textbf{ 134}, (2006),  no. 1, 39--50.


\bibitem{CRS} Cibils, C.; Redondo M. J.; Solotar, A. The intrinsic fundamental group of a linear category. To appear in Algebras and Representation Theory. DOI: 10.1007/s10468-010-9263-1.

\bibitem{CRS2} Cibils, C.; Redondo M. J.; Solotar, A. Connected gradings and fundamental group. Algebra Number Theory \textbf{ 4} (2010), no.
    5, 625--648.

\bibitem{ciso} Cibils,C.; Solotar, A. Galois coverings, Morita equivalence and smash extensions of categories over a field. Doc. Math. \textbf{ 11} (2006), 143--159.

\bibitem{CP} Clifford, A.H.; Preston, G.B. The algebraic theory of semigroups Volume 1. American Mathematical Society, Providence, R.I, 1977.

\bibitem{dodo}Douady, R.; Douady, A. Alg\`ebre et th\'{e}ories galoisiennes.  Paris: Cassini. 448 p. (2005).


\bibitem{ga} Gabriel, P. The universal cover of a representation-finite algebra. Representations of algebras (Puebla, 1980), 68--105, Lecture Notes in Math. \textbf{ 903}, Springer, Berlin-New York, 1981.

\bibitem{GAZI} Gabriel, P.; Zisman, M. Calculus of fractions and homotopy theory. Ergebnisse der Mathematik und ihrer Grenzgebiete, Band \textbf{35} Springer-Verlag New York, Inc., New York, 1967

\bibitem {GEMA} Gelfand, S. I.; Manin, Yu. I. Homological algebra. Springer-Verlag, Berlin, 1999.



\bibitem{H} Hatcher, A. Algebraic topology. Cambridge University Press, Cambridge, 2002. xii+544 pp.

\bibitem{HI} Higgins, P. J. Categories and groupoids. Repr. Theory Appl. Categ. \textbf{7} (2005). http://www.tac.mta.ca/tac/reprints/articles/7/tr7abs.html

\bibitem{KASC} Kashiwara, M.; Schapira, P. Sheaves on manifolds.
 Grundlehren der Mathematischen Wissenschaften [Fundamental Principles of Mathematical Sciences], 292. Springer-Verlag, Berlin, 1990.



\bibitem{le}Le Meur, P. The universal cover of an algebra without double bypass, J. Algebra  \textbf{ 312}  (2007),  no. 1, 330--353.

\bibitem{le1}Le Meur, P. The fundamental group of a triangular algebra without double bypasses. C. R. Math. Acad. Sci. Paris 341 (2005), 211--216.

\bibitem{le2}Le Meur, P. Rev\^{e}tements galoisiens et groupe fondamental d'alg\`ebres de dimension finie. Ph.D. thesis, Universit\'{e} Montpellier 2 (2006). \texttt{http://tel.archives-ouvertes.fr/tel-00011753}


\bibitem{MacLane} Mac Lane, S. Categories for the working mathematician (Second edition).  Springer-Verlag, New York, 1998.

\bibitem{MP} Mart\'\i nez-Villa, R.;  de la Pe\~na, J. A. The universal cover of a quiver with relations. J. Pure Appl. Algebra \textbf{ 30} (1983), 277--292.


\bibitem{May} May, J.P. A concise course in algebraic topology.  The University of Chicago Press, Chicago, 1999.

\bibitem{MIN} Minian, E.G. Cat as a $\Lambda$-cofibration category. J. Pure Appl. Algebra \textbf{167}, 301-314.
2002.

\bibitem{Mi} Mitchell, B. Rings with several objects. Adv. Math. \textbf {8} (1972), 1--161.

\bibitem{OJ} Ojeda B\"{a}r, N. Revestimientos categ\'{o}ricos, simpliciales y topolog\'{\i}as de Grothendieck. T\'{e}sis de Licenciatura, Departamento de Matem\'{a}tica, Facultad de Ciencias Exactas y Naturales, Universidad de Buenos Aires (2006).\\
   cms.dm.uba.ar/academico/carreras/licenciatura/tesis/ojeda.pdf


 \bibitem{QU}   Quillen, D. Higher algebraic $K$-theory. I. Algebraic $K$-theory, I: Higher $K$-theories (Proc. Conf., Battelle Memorial Inst., Seattle, Wash., 1972), pp. 85--147. Lecture Notes in Math., \textbf{341}, Springer, Berlin 1973.


\bibitem{ta} Tanaka, K. A model structure on the category of small categories for coverings
	arXiv:0907.5339v1


\end{thebibliography}
 \end{document}